\documentclass[a4wide]{amsart}
\usepackage{mathptmx} 
\usepackage{url} 
 \usepackage{amsthm}
 \newtheorem{prop}{Proposition} 
 \newtheorem{defi}{Definition} 
\usepackage{amssymb}
\usepackage{amscd} 

\newcommand{\Hom}{\mathop{\mathrm{Hom}}}
\usepackage{prooftree} 
\newcommand\seq\vdash 
\newcommand\eee{\mathbf{e}}
\newcommand\ttt{\mathbf{t}}
\newcommand\type[1]{^{#1}}
\newcommand\et{\mathop{\&}} 
\newcommand\imp\rightarrow 
\newcommand\Sh{\mathbf{Sh}}
\newcommand\ccc{CCC} 

\newcommand\fl\rightarrow 

\newcommand\ma[1]{``\emph{#1}"}

\begin{document}

\author{Jean Gillibert, Christian Retor{\'e}} 
\title[Category theory, logic and formal linguistics: 
some connections, old and new]{Category theory, logic and formal linguistics:\\ 
some connections, old and new} 
\address{IMB, LaBRI --- Universit\'e de Bordeaux}
\maketitle 
\begin{abstract}
We seize the opportunity of the publication of selected papers from the  \emph{Logic, categories, semantics}  workshop in the \emph{Journal of Applied Logic}  to survey some current trends in logic, namely intuitionistic and linear type theories,  that interweave categorical, geometrical and computational considerations. We thereafter present how these rich logical frameworks 
can model the way language conveys meaning. 
\newcommand\sep\quad 



\end{abstract}



\section{A seminar and workshop on category theory, logic and linguistic applications}

The present issue of the \emph{Journal of Applied Logic} gathers a selection of papers presented at a workshop \emph{Logic, categories, semantics} held in Bordeaux in November 2010. 
This workshop was organised as a fitting conclusion to
the activities of a weekly reading group a weekly reading group called \emph{Sheaves in logic and in geometry} in 2009/2010
and \emph{Logic, categories, geometry} in 2010/2011. 

Those activities are common to the maths and computing departments of the University of Bordeaux (IMB-CNRS and LaBRI-CNRS). Although he did not contribute to this introduction, we must thank Boas Erez (IMB, Bordeaux) who was an enthusiastic participant and speaker in our reading group. We would like to thank for their scientific and financial support:   our tow departments, (Institut de Math\'ematiques de Bordeaux, Laboratoire Bordelais de Recherche en Informatique), INRIA Bordeaux Sud-Ouest, the ANR project LOCI, and the project ITIPY (Aquitaine Region) and the French mathematical society (SMF). 

Unfortunately, the material in this issue of the \emph{Journal of Applied Logic} does not cover all the great talks we had the privilege  to hear, 
but one can find published material that more or less cover these talks (and these references themselves  include further references): 
\begin{itemize}
\item 
Pierre Cartier 
(IHES, Bures)	\emph{Sur les origines de la connexion entre logique intuitionniste, cat\'egories et faisceaux}, 
cf. \cite{Cartier1979bourbaki}; 
\item 
Jean-Yves Girard 
(CNRS, IML, Marseille)	\emph{Interdire ou r\'efuter? Le statut ambigu de la normativit\'e}, cf. 
\cite{Girard2012normativity}; 
\item 
Paul-Andr\'e Melli\`es (CNRS, PPS, Paris)
\emph{Logical proofs understood as topological knots}, cf. \cite{Mellies2012lics}; 
\item 
Michael Moortgat 
(Universiteit Utrecht)	\emph{Continuation semantics for generalized Lambek calculi}, cf. \cite{MoortgatMoot2013grishin};  
\item
Carl Pollard (Ohio State University, Columbus)
\emph{Remarks on categorical semantics of natural language}, cf. \cite{Pollard2011lacl}.
\end{itemize} 

\section{Intuitionistic logic} 
\label{intui} 

Type theory and categorical logic, which are at the hear of this issue, all rely on intuitionistic logic.
Even  linear logic, which is also discussed by two articles of this issue,  may be viewed as a refinement of intuitionistic logic. 

It is quite a challenge to explain briefly what intuitionistic logic is and how it relates to other logical trends. 
Regarding the historical development of logic, we refer to the standard \cite{KK86}
and regarding a standard logic master course we refer the reader to the standard 
textbook  \cite{vanDalen2013}
which includes a chapter on intuitionistic logic. 

Logic is concerned with``truth", and 
 let us vaguely say   that ``truth" holds or not 
of formulae that  are formalisations of common language sentences, or of mathematical statements, 
and as such logic has at least two facets. 
A formula $F$ can be true because it is correctly derived from 
formulae that are already established or assumed to be true, and this is the proof theoretical view of truth that started with Aristotle.  The same formula 
$F$ can also  be true in all or in some situation once the symbols in the formula are properly interpreted: that's the model theoretical view of truth, which is much more recent. 
Besides these two well established viewpoints, there exists a slightly different view of ``truth" with Ancient origins, dialectics which emerge from the interaction
between proofs and refutations.  
Nowadays this view culminates in game theoretical semantics, and ludics,
and the presentations by  J.-Y. Girard  and P.-A. Melliès during our workshop developed this interactive view. \cite{Girard2012normativity,Mellies2012lics}

An important theorem, due to G\"odel   as  
an impressive number of the fundamental results logic, is that for first order classical logic, 
the two notions of truth coincide: a formula is true in any model if and only if it can be proved. 
This result admits a a more striking formulation as: a formula $C$  can be derived from formulae $H_1,\ldots,H_n$ 
if and only if any model satisfying $H_1,\ldots,H_n$ satisfies $C$ as well. 
Such a result holds for classical propositional logic and for classical first-order logic, also known as predicate calculus. 

Intuitionistic logic comes from doubts that appeared during the crisis of the foundations of mathematics,
between the XIX and XX centuries.
These doubts concerned existential statements in infinite sets. For instance it is admittedly a bit strange, especially in an infinite universe,  that 
$\exists x (B(x) \implies \forall y B(y))$, although it is absolutely correct in classical logic using principles like  $U\implies V\equiv \lnot U \lor V$, $\lnot\lnot A\equiv A$, $\lnot \forall x A(x)\equiv \exists x \lnot A(x)$. 
According to intuitionistic logic,  \emph{tertium non datur}, or any of the principles that are equivalent to it,  like \emph{reductio ad absurdum} or \emph{Pierce's law}, should be left out from deduction principles.

When restricting ourselves to intuitionistic logic both facets of truth are modified: 
\begin{itemize}
\item In order to maintain the completeness results, models need to be more complex: the usual models ought to be replaced with structured families of classical models, like Kripke models or (pre)sheaves, cf. section \ref{sheaves}. 
\item 
Intuitionistic proofs are more interesting: as opposed to classical proofs, they can be viewed as algorithms (as programs of a typed functional programming  language) 
or as morphisms in a category (which is not a poset, i.e. with many morphisms from an object to another). 
\end{itemize} 

Each of these two aspects brings connections with category theory. The first one yields models  of higher order intuitionistic logic, e.g. with (pre)sheaves of $\mathcal{L}$-structures (classical models, roughly speaking).  The second one yields models that interpret proofs as morphisms in a cartesian closed category. 

\section[A historical connection between logic and geometry: (pre)sheaves, toposes and  categorical models of 
intuitionistic logic]{A historical connection between logic and geometry:\\ (pre)sheaves, toposes and  categorical models of 
intuitionistic logic}
\label{sheaves} 

Pierre Cartier  reminded us in the inaugural lecture of the workshop that the connection between intuitionistic logic and topos theory appeared in the sixties. We refer the reader to
 \cite{Cartier1979bourbaki} for a survey, and to 
 \cite{MacLaneMoerdijk1992}
 for a thorough treatment, and to \cite{Mac71} for an excellent textbook on category theory. 

Inspired by the work of Eilenberg and Mac Lane on categories, Lawvere started in 1964 to introduce a categorical approach to  set theory, and Lambek was looking at categories as deductive systems, the 1972 book edited by Lawvere \emph{Toposes, Algebraic Geometry, and Logic}  gives a pretty good idea of the landscape at the end of the 60s. 
\cite{Lawvere72}

At the same time, Grothendieck was introducing groundbreaking concepts  in order to define a topology on an algebraic variety that would play a similar role to the topology of the complex points, namely the \'etale topology. The first step in Grothendieck's construction is the notion of site, that is, a category endowed with a notion of covering that is reminiscent of open coverings in topology.  
The concept of site does not exactly capture the`topology" of a situation, since different sites can give rise to the same category of sheaves.  Thus Grothendieck chose to take a category of sheaves of sets over a given site as a notion of generalized topological space, and he called such a category a topos
But another important idea of Grothendieck was somehow more surprising: the \'etale site is just a step in the construction, the final notion to study is the category of sheaves of sets on the site, that he calls a topos. This category has many properties in common with the category of sets, and possesses a good notion of localisation. These radically new ideas were first developed by Grothendieck and his group in \cite{SGA4}. 


Taking into account this new approach, Lawvere and Tierney  in \cite{Lawvere1970congres} gave before 1970 a new definition of topos (elementary toposes), which does not refer to a site. One of the avantages of this theory is that each topos contains an object $\Omega$, that plays the role of all possible truth values, and has the (internal) structure of a Heyting algebra. Hence intuitionist logic naturally appears as being the internal logic of the topos.

This connection with intuitionistic logic was formulated by Lawvere at the same time 
 \cite{Lawvere1970congres}. 
Shortly afterwards, Lawvere and Tierney  in \cite{Lawvere72} discovered that models of set theory corresponding to proofs of the independence of the axiom of choice and continuum hypothesis by Cohen's method of forcing \cite{cohen1966set} were in fact elementary toposes see e.g. \cite{MacLaneMoerdijk1992,Borceux1994hdbk3}. 


On the other hand, the internal logic of the topos shed a new light on categories of sheaves. 
Namely, one can handle sheaves as if they were sets with elements but with respect to the internal language, but when reasoning about these elements one should follow should follow the rules of intuitionistic logic. 

In order to fix ideas, let us recall that a \emph{category} $\mathcal{C}$ is defined by the following data: (a) a class of objects; (b) a class of arrows (or morphisms) whose domain and codomain are objects of $\mathcal{C}$; (c) a partial binary operation on arrows called composition, required to be associative; (d) every object is endowed with an identity arrow.  If for any objects $A$ and $B$ of $\mathcal{C}$, the class $\Hom(A,B)$ of all morphisms from $A$ to $B$ is  a set, the category is said to be \emph{locally small}, and most of the useful categories are locally small. 
If the class of objects and of morphisms of $\mathcal{C}$ are sets, $\mathcal{C}$ is said to be \emph{small}. 

It is easy to define morphisms of categories, which for historical reasons are called \emph{functors}.

A category can be seen as a directed graph with a composition operation on arrows. The most basic examples of categories are the category of sets, the category of groups, the category of $K$-vector spaces where $K$ is a field, the category of topological spaces, etc.
Roughly speaking, categories allow to handle most kind of structural questions in mathematics.

We will now focus our attention on sheaves, which are central in this story. Let $X$ be a topological space. A presheaf of sets is a contravariant functor from the category of open subsets of $X$ (a poset, actually) to the category of sets\footnote{Explicitely, a presheaf $A$ is a family of sets $A(U)$ (where $U$ runs through all open subsets of $X$) together with restriction maps $A(U)\to A(V)$ for any inclusion $V\subseteq U$ between open subsets of $X$. Of course, these restriction maps can be composed, that is, if $W\subseteq V \subseteq U$ then the restriction map $A(U)\to A(W)$ is equal to the composition $A(U)\to A(V)\to A(W)$.}.

A morphism of presheaves $\phi:A\to B$ is given by a family of maps $\phi_U:A(U)\to B(U)$ (where $U$ runs through all open subsets of $X$) that is compatible with restriction maps in the following sense: for any inclusion $V\subseteq U$ between open subsets of $X$, the following diagram commutes
$$
\begin{CD}
A(U) @>\phi_U>> B(U) \\
@VVV @VVV \\
A(V) @>\phi_V>> B(V) \\
\end{CD}
$$
where vertical maps are the restrictions induced by the inclusion $V\subseteq U$.

It is worth noticing that (pre)sheaves of $\mathcal{L}$-structures of some theory $\mathcal{T}$ over a topological space --- whose open sets can be viewed as a  poset category that is a Heyting algebra ---   define a notion of model that is complete for intuitionistic logic. 
An $\mathcal{L}$-structure is a set with an interpretation of the first order language $\mathcal{L}$, and the morphisms between $\mathcal{L}$-structures, in particular the restrictions morphisms of the sheaf, are functions that preserves function symbols and the 
truth of atomic formulae --- hence they are more general than the usual elementary morphisms of model theory which preserve the truth of all formulae. 
This construction due to Joyal is related to forcing and is a particular case of Kripke models \cite{Kripke59jsl}
and therefore is known as Kripke-Joyal forcing. \cite{BoileauJoyal1981} The proof of completeness of (pre)sheaf models can be inferred from results in \cite{TD88v2}; recently, 
 Ciardelli obtained a short and direct proof of this completeness result \cite{CIardelli2011tacl}. 

A sheaf of sets is a presheaf $\mathcal{F}$ which satisfies the following condition: for any open subset $U$ of $X$ and any family $(U_i)$ of open subsets of $X$ such that $U=\cup U_i$, there is an equaliser diagram
$$
\mathcal{F}(U)\rightarrow \prod_i \mathcal{F}(U_i) \rightrightarrows \prod_{j,k} \mathcal{F}(U_j\cap U_k)
$$
(this means that the first map is injective and its image is exactly the set where the two other maps coincide\footnote{The first map is induced by inclusions $U_i\subseteq U$, and the two other maps are respectively induced by inclusions $U_j\cap U_k\subseteq U_j$ and $U_j\cap U_k\subseteq U_k$.}).

In other words, a ``compatible'' family of local sections of $\mathcal{F}$ can be glued---in a unique way---into a global section of $\mathcal{F}$ i.e. a $U$-section of $\mathcal{F}$. This is exactly what one needs in order to perform ``local to global'' yoga.

By definition, a morphism of sheaves is just the same as a morphism of presheaves. Hence the category of sheaves on $X$, that we denote by $\Sh(X)$, is a full subcategory of the category of presheaves. In order to describe its properties, we introduce the notion of cartesian closed category (\ccc).

\begin{defi}
\label{ccc} 
A category is said to be cartesian closed whenever: 
\begin{enumerate}
\item \label{terminal} It has a terminal object here noted $\mathbf{1}$ (it can be noted $T$);
\item  \label{product}  Any two objects $A$ and $B$ have a product written $A\times B$; 
\item \label{exponential} Any two objects $A$ and $B$ have an exponential noted $B^A$ or $A\fl B$.
\end{enumerate} 
\end{defi} 
In more detail, \ref{terminal} a terminal object is an object $\mathbf{1}$ such that for any object $A$ there exists a unique arrow $A\to \mathbf{1}$; \ref{product}  if $A$ and $B$ are two objects, the product $A\times B$ (if it exists) is an object such that, for all $Y$, there is a bijection $\Hom(Y,A)\times\Hom(Y,B)\simeq\Hom(Y,A\times B)$ which is natural in $Y$ ; 
\ref{exponential} if $A$ and $B$ are two objects, the exponential noted $B^A$  (if it exists) is an object such that, for all $Y$, there is a bijection $\Hom(Y\times A,B)\simeq\Hom(Y,B^A)$ which is natural in $Y$. Another limit is useful, \emph{equaliser}: given tow morphisms $f$ and $g$ from $A$ to $B$, an equaliser of $f$ and $g$ is a morphism $e$ from $C$ to $A$ such that $f\circ e= g\circ e$ and for any morphism $d$ from $D$ to $A$ with $f\circ d= g\circ d$ there is a unique $k$ from $D$ to $C$ such that $e \circ k= d$.

Let us check that $\Sh(X)$ is cartesian closed. The constant sheaf $\mathbf{1}$ defined by $\mathbf{1}(U)=\{*\}$ for all $U$ is clearly a terminal object in $\Sh(X)$. The product of two sheaves, as well as the equaliser
 of sheaf morphisms with the same domain and codomain exist, hence all finite limits exist in the category of sheaves. 
Finally, exponential objects exist: for any two sheaves $A$ and $B$, we let $B^A$ be the sheaf $U\mapsto \Hom(A|_U,B|_U)$, where $\mathcal{F}|_U$ denotes the restriction of $\mathcal{F}$ to $U$\footnote{This sheaf $U\mapsto \Hom(A|_U,B|_U)$ is often called ``internal hom''.}.

More generally, the category $\Sh(X)$ inherits many properties from the category of sets, which is not surprising if we consider a sheaf as being a ``continuous family of sets''. This leads us to the next topic, namely elementary toposes.

\begin{defi}\label{topos}
According to Lawvere and Tierney,  an \emph{elementary topos} is a category $E$ satisfying the following properties:
\begin{enumerate}
\item $E$ has finite limits;
\item $E$ has exponentials;
\item $E$ has a subobject classifier.
\end{enumerate}
\end{defi} 

The first two conditions imply that a topos is a cartesian closed category. In fact, the main difference between \ccc\  and toposes is the existence of a subobject classifier.

Let us recall that $\Omega$ is a subobject classifier if, for any object $A$ of $E$, there is a natural isomorphism between subobjects of $A$ and the set of maps $A\to \Omega$. More precisely, $\Omega$ comes with a canonical map $\mathbf{1}\to\Omega$ (where $\mathbf{1}$ is the terminal object of $E$) which satisfies the following universal property:  for any monomorphism $f:B\to A$ there exists a unique morphism $g:A\to \Omega$ such that the square
$$
\begin{CD}
B @>>> \mathbf{1}\\
@VfVV @VVV \\
A @>g>> \Omega\\
\end{CD}
$$
is a fibered product or pullback (roughly speaking, $B$ is the inverse image of $\mathbf{1}$ by $g$).

In the category of sets, $\Omega=\{0,1\}$ (Boolean truth-values).

In the category of sheaves, the subobject classifier is defined as follows: for any open subset $U$ of $X$, we let $\Omega(U)$ to be the set of all open subsets of $U$. Roughly speaking, an assertion inside the topos $\Sh(X)$ may have an intermediate truth value, which from the viewpoint of an open subset $U$ is the open subset of $U$ where the assertion is true.

\emph{In this issue, the paper \emph{Continuity and geometric logic} by Steve Vickers emphasises the connection with topology by exploring geometric logic, that is many sorted logic with finite conjunctions and arbitrary disjunctions (also existential quantification). Given that classical models of a geometric theory 
$\mathbf{T}$  are not complete, one has to look models of $\mathbf{T}$ inside a topos, and there exists a classifying topos whose category of T models is universal. This is a point-free
manifestation of the ``space of models of $\mathbf{T}$ ".}

\section{Intuitionistic proofs, typed lambda calculus and cartesian closed categories}
\label{lambda} 

A second link between intutionistic logic and category theory is the following. 
Proofs in (propositional) 
intuitionistic logic can be viewed as typed lambda terms, that are functional programs.
The main idea  is that a proof of $A\imp B$ maps proofs of $A$ to proofs of $B$,
hence it is rather natural to interpret a formula (with $\land, \imp$) by the sets of its proofs.
Functions can be applied to arguments of the right type, and composed. 
To compute such a term, one substitutes the occurrences of the bound variable used 
to define the function with the argument of the function: 
$$(\lambda x^A.\ t^B)(u^A)\stackrel{\beta}{\rightsquigarrow}t[x^A{:}=u^A]$$ 
Values are normal terms i.e. terms that cannot be reduced any further. 

In figure \ref{curry} we give the proof rules and the corresponding proof terms, that are simply typed lambda terms, a good reference being \cite{GLT88} and \cite{LS86} on the categorical aspects. 

\newcommand\veczz{z_1{:}Z_1,\!...,z_p{:}Z_p}
\newcommand\veczzs{z_{\sigma(1)}{:}Z_{\sigma(1)},\!...,z_{\sigma(p)}{:}Z_{\sigma(p)}}
\newcommand\vecyy{y_1{:}Y_1,\!...,y_k{:}Y_k}

\begin{figure} 
\label{curry} 

\begin{center}
\begin{prooftree} 
axiom
\justifies 
x{:}A\seq x{:}A
\end{prooftree} 
\hspace*{10em} 
\begin{prooftree} 
\veczz \seq t{:}B
\justifies 
\veczzs \seq t{:}B
\using \begin{array}{c}exchange\\ \scriptstyle (\sigma: permutation)\end{array}  
\end{prooftree} 
\vspace*{3ex}

\begin{prooftree} 
\veczz, x{:}A,  x{:}A \seq t{:}B
\justifies 
\veczz, x{:}A \seq t{:}B
\using contraction 
\end{prooftree} 
\hspace*{6em} 
\begin{prooftree} 
\veczz, \seq t{:}B
\justifies 
\veczz, x{:}A \seq t{:}B
\using weakening  
\end{prooftree} 
\vspace*{3ex}

\begin{prooftree} 
\veczz, x{:}A \seq t{:}B
\justifies 
\Gamma, x{:}A \seq (\lambda x^A. t){:}A\fl B
\using abstraction\ /\ hypothetical\ reasoning 
\end{prooftree} 
\vspace*{3ex}

\begin{prooftree} 
\veczz,  \seq f {:}A\fl B \qquad \vecyy \seq u{:}A
\justifies \veczz,\vecyy \seq f(u){:}B 
\using application\ /\ modus\ ponens 
\end{prooftree} 
\end{center} 

\caption{The Curry-Howard isomorphism \cite{How69}. Upper case letters stand for propositions. 
On the left handside of ``$\seq$" one only has type assignments for variables, like $x:X$.
On the right hand side, before the ``:" there is a lambda term which encodes the proofs. 
Leaving out terms, that are on the left hand side of ``:" one get a standard natural deduction system for proposition logic as introduced by Gentzen. \cite{Gen34,Gen34b}
These rules infer a sequent from one or two sequents a sequent being a sequence of hypotheses and a formula on the right hand side, $Z_1,..,Z_p\seq C$ meaning: $(Z_1\& \cdots Z_p)\fl C$.} 
\end{figure} 

Functions in this setting are algorithms (that always converge) that represent  total recursive functions, the ones 
whose totality can be proved within the underlying logical system.  They are defined intentionally, how could one view them as usual functions? 
Functions can be viewed as morphisms in a cartesian closed category, cf. definition \ref{ccc} above --- in this logical setting,  $B^A$ is rather written $A\imp B$. 
Observe that the logical rules match categorical isomorphisms, for instance,  
$$A \imp (B \imp C) \cong (A \times B) \imp C$$  
Closed terms (i.e. terms without free variables) of type $U$ also can be viewed as maps 
from the terminal object $\mathbf{1}$ to $U$ ---  this view is supported by the fact, that $\mathbf{1}=\{*\}$ in the category of set. 

A function in the model interprets a function as a static notion (a ``set" of pairs $\langle x,f(x)\rangle$)  and not as a (functional) algorithm: indeed whenever a term/proof $t$ reduces to $t^\circ$ its interpretation as a morphism in the category is the same. 
%


Given that functions from $A$ to $B$ are definable by terms/proofs,  the set $\Hom(A,B)$ should be countable. Therefore, ``elements" in $B^A$, that are arrows from in $\textbf{1}$,  should be countable as well,
and this imposes a drastic restriction on the morphisms we consider in the category. The study of such models of intuitionistic proofs lead to linear logic as  we shall see in section \ref{linear}.

A result by Joyal (unpublished see e.g. \cite{LS86,Girard2011blindspot}) shows that all proofs of a given formula are equivalent in classical logic. Although there also exists a logical, syntactical analogous result by Girard and Lafont (see figure \ref{lafont}), Joyal  proved  a category theoretical formulation of this statement: 

 \begin{prop}
In  a cartesian closed category interpreting classical logic 
between any two objects (formulae)  $A$ and $B$ 
there is 
at most one morphism from $A$ to $B$ (at most one interpretation of proofs of $A\seq B$). 
 \end{prop} 
 \begin{proof} 
 The interpretation of classical logic needs an initial object $\mathbf{0}$ (false) and a negation $\lnot A=(A\fl\mathbf{0})$ which is involutive i.e. $A\cong ((A\fl \mathbf{0}) \fl \mathbf{0})$. 
 To see that in such a situation there exists at most one arrow between any two objects, 
 let us first see that in a \ccc\  with an initial object $\mathbf{0}$ when there is a map from $Z$ some object to $\mathbf{0}$, then $Z$ is initial as well: 
 \begin{enumerate}
 \item \label{Zx0} 
\emph{If $\mathbf{0}$ is initial, then is $\mathbf{0}\times Z$ is initial as well.} Indeed, in a \ccc\ , maps from 
$\mathbf{0}\times Z$ to any object $U$ are in bijection with maps from $\mathbf{0}$ to $U^Z=[Z\fl U]$. 
\item
\label{Zinit} 
If there exists an arrow $f$ from $Z$ to $\mathbf{0}$ then $Z$ is isomorphic to $\mathbf{0}\times Z$ and therefore $Z$ is initial. 
Indeed, $\pi_2\circ \langle f, Id_Z\rangle=Id_Z$ (by definition of $\pi_2$) and $\langle f, Id_Z\rangle\circ \pi_2=Id_{\mathbf{0}\times Z}$ because it goes from $\mathbf{0}\times Z$ to itself, which is initial because of \ref{Zx0} above. 
\end{enumerate} 
Maps from $A$ to $B$ correspond to maps from $A$ to $\neg\neg B \cong
(B \to \mathbf{0}) \to \mathbf{0}$ and thus to maps from $A \times (B \to \mathbf{0})$ to $\mathbf{0}$.
But from  \ref{Zinit}  it follows that such a map is an isomorphism and there is
precisely one such isomorphism since $\mathbf{0}$ is initial.
\end{proof}

\begin{figure} 
\begin{center} 
\begin{prooftree} 
\begin{prooftree} 
 \begin{prooftree} 
 \begin{prooftree} 
\delta_1
\leadsto 
 \seq F
 \end{prooftree} 
 \justifies 
 \seq F,K
 \using weakening 
 \end{prooftree} 
 \hspace*{5em} 
 \begin{prooftree} 
 \begin{prooftree} 
 \delta_2
\leadsto 
  \seq F
 \end{prooftree} 
 \justifies 
 \seq F,\lnot K
 \using weakening 
 \end{prooftree} 
 \justifies 
 \seq F, F
 \using cut
 \end{prooftree} 
 \justifies 
 \seq F
 \using contr.
 \end{prooftree} 
 \end{center} 
\caption{This proof defined from  $\delta_1$ and  $\delta_2$--- that are any two proofs of the same formula $F$ ----  reduces to $\delta_1$ or to $\delta_2$ followed by the exactly the same rules. Hence any two proofs of the same formula are equivalent in classical logic.} 
\label{lafont} 
\end{figure}

This result, simple enough to be presented in this introduction, clearly shows that intuitionistic logic 
is the natural logic of the \ccc\  interpretation of proofs --- in the previous section we already saw that intuitionistic  also is the logic of toposes which generalise set theory. 

\section{Martin-L\"of type theory and identity types} 
\label{mtt} 

One may further refine the operations (exponential, product) 
on the simple types of Church \cite{Church1940types} of the previous section \ref{lambda} as linear logic defined in next section \ref{linear} does, 
or consider a wider set of  types as Martin-L\"of type theory does.

Actually,  the most prominent feature of Martin-L\"of type theory is that there are
types which \emph{depend} on  terms  as e.g. $\mathrm{Vect}(n)$ for $n:N$, 
the type of vectors of length $n$. Propositions are considered as
types and, accordingly, predicates may be considered as dependent types.
The most important example of dependent types will be identity types
as discussed subsequently.  

Martin-L\"of type theory provides the
necessary technology to deal with the standard mathematical concept of
a family of sets. Moreover a dependent type can be considered as as
predicate; for example if we are given a dependent type $B(x)$ with $x:A$ (a family of types indexed by the
type $A$), we can think of \(A\) as a set and \(B(x)\) as a family of
sets, but in addition, if we consider \(B(x)\) as a predicate, given a hypothetical element \(a\) of \(A\)
and a hypothetical \(b\) in \(A(a)\) we can think of \(b\) as a
\emph{proof witness} that shows the truth of \(B(a)\).

Type theory as formalized by Martin-L\"of gives us access to the construction of
\begin{itemize} 
\item 
the type $(\Pi x:A) B(x)$  whose intuitive elements should be seen as functions $f$ that map a term  
$a:A$ to a term $f(a) : B(a)$. In case $B$ does not depend on $x$, the type
$(\Pi x:A) B(x)$ corresponds to the type $A\fl B$ of simply typed lambda 
calculus. When we see dependent types as predicates, we now have access to
the universal quantifier.
\item the type $(\Sigma x:A) B(x)$ whose elements should be seen are pairs $(a,b)$ with
$a : A$ and $b:B(a)$. If $B(x)$ does not depend on $x$, the type  
$(\Sigma x:A) B$ corresponds to the type $A\times B$ of simply types lambda 
calculus, and we also have access to the existential quantifier when
we think of a type as a predicate.
\end{itemize}

The presence of dependent types forces us to consider the following judgements

\begin{tabular}{ll} 
$\Gamma\seq Valid$ & The context $\Gamma$ is valid
\\ 
$\Gamma\seq A$  &in the context $\Gamma$, $A$ is a type. 
\\ 
$\Gamma\seq A=B$  & the types $A$ and $B$ are equal 
\\ 
$\Gamma\seq t:A$  &in the context $\Gamma$, $t$ is of type $A$ 
\\
$\Gamma\seq t=u: A$  & $t$ and $u$ are equal terms of type $A$. 
\end{tabular} 

This notion of equality is called \emph{judgemental} and can be decided
by normalisation. But it cannot be combined into propositions for which reason
one introduces the notion of \emph{propositional} equality. If $t,s :A$
one may form the types $Id_A(t,s)$ whose elements will be proofs that $t$ and
$s$ are equal. Reflexivity is given by the constructor $r_A(t) : Id_A(t,t)$
for $t:A$. As usual in type theory a constructor is accompanied by an
\emph{eliminator} which in case of identity types happens to be called $J$
and can be described as follows. Given a family of types $x,y:A, z:Id_A(x,y) 
\seq C$ and a term $x,y:A, z:Id_A(x,y) \seq d : C(x,y,z)$ one may consider
the term $x,y:A, z:Id_A(x,y) \seq J((x) d,z) : C(x,y,z)$ whose meaning is
provided by the reduction rule $ J((x) d,r_A(t)) = d[t/x]$. This is in 
accordance with functional programming where functions are defined by pattern
matching \emph{aka} structural recursion.

This is how equality is represented in \emph{intensional} type theory of 
\cite{ML73}. In \cite{ML84} judgemental and propositional equality get
identified via a so-called \emph{equality reflection} rule allowing one
to conclude $t =s : A$ from $p : Id_A(t,s)$. This, however, renders type
checking undecidable and, therefore, has been abandoned. Current proof checkers
as e.g. the widely used system \(\mathtt{Coq}\)~\cite{BC2004coq} are based on intensional type
theory. This is sometimes inconvenient but allows one to avoid the manipulation and storing of derivation trees for typing judgements as in the much older
\(\mathtt{NuPrL}\) system.

As said above in section \ref{sheaves}, toposes admit an internal logic in form
of intuitionistic higher order arithmetic. In this logic propositions are
not represented as types but as subobjects of the terminal object $1$. 
Accordingly, predicates on a type $A$ are represented as monomorphisms into $A$.
On the other hand, a good interpretation of Martin-L\"of type theory
needs a predicate on a type \(A\) to be represented by a morphism
\(B\to A\) which is more
general than just a monomorphism, since we want the ability to distinguish between
proof witnesses.  Thus we need a a distinguished class of morphisms 
(often called \emph{display maps}), which are subject to a few simple
conditions, the most important of which being that the pullback of a
display map with codomain \(A\) by \emph{any} morphism with same
codomain is again a display map.  This corresponds to being able to
substitute a term for a variable in a predicate.  It is then possible to
interpret the $\Pi$-types when a right adjoint always exists to
that pullback functor.  When this happens we have a \emph{display
  category with (dependent) products}, and when all maps of the
category can be used
as display maps, we have a
\emph{locally cartesian closed} categories.
There are many display categories with dependent products, but it
turned out to be harder to find examples that have a good, nontrivial
(i.e., non collapsed) versions of the intensional identity type. For example
locally cartesian closed categories are guaranteed to give trivial identity
types.  The first good example was given by M.~Hofmann and Th.~Streicher in~\cite{gpdtt},
where types were interpreted as groupoids, display maps as
\emph{fibrations of groupoids} (very easy to define), and $Id_A(x,y)$ as the set of
morphisms (necessarily isomorphisms) from $x$ to $y$.  This model was
only nontrivial in dimension one, in the sense that iterating the
identy type would land us on the old, trivial, nonintensional interpretation where identity
is just the diagonal.

It was realized that if a model of dependent types has an intensional
identity type, then we are very close to being able to do topology
(especially, homotopy theory) in
it.  This is because in such an identity type, we can view an
intuitive element of \(ID_{A}(x,y)\) (a proof that \(x\) is equal to
\(y\)) as a \emph{path from \(x\) to \(y\),} and we are able to do
many constructions of homotopy theory, with the additional bonus of
dependent products.

Thus inspirations for good models of identity types were bound to come from algebraic
topology, and this is the case for the first completely non-degenerate model, where paths in higher dimensions
(higher homotopies) do not collapse, and which was found
independently by Streicher and Voevodsky around 2006 (see e.g. \cite{Voevodsky2006lambda,Streicher2006uppsala}. 
The base category is the presheaf topos of \emph{simplicial sets}, the
most studied of all models in combinatorial topology, and families of types
(i.e., display maps) are interpreted as \emph{Kan fibrations}, which
forces ordinary, non-dependent types to be~\emph{Kan complexes}. 
As expected, the identity type $Id_A(x,y)$ is the standard
construction for the \emph{path
  object} in the sense of the usual Quillen model structure on
simplicial sets.\cite{quillen1967homotopical}

Another inspiration from topology, namely in the theory of fiber
bundles, has led V.~Voevodsky~\cite{ACV2013} to propose the
\emph{Univalence Axiom} in what is now called \emph{Homotopy Type Theory}
(HoTT), which applies to universes, like those found in
\(\mathtt{Coq}\) and those defined by Martin-L\"of, and  can be viewed
as the statement that in a ``good'' universe \emph{isomorphic types are equal}.

The present issue contains two papers about categorical models of identity 
types, both having links to homotopy type theory:
\begin{itemize}
\item \emph{Fran\c{c}ois Lamarche in his paper \emph{Modelling Martin-L\"of Type Theory in Categories} proposes 
a simple model of Martin-L\"of type theory that includes both dependent products and the identity types.}
\item \emph{Thomas Streicher in the next paper describes \emph{A Model of Type Theory in Simplicial Sets} in which the Voevodsky Univalence axiom holds.} 
\end{itemize}

Although little work has been done in this direction,  categorical
models of type theory should be very appealing from the viewpoint of natural language semantics, to be evoked in section
\ref{applications}.

Indeed, natural language semantics, philosophy of language, analytic philosophy 
has always been interested in models (classical models,  possibles worlds, situations) and nowadays,  natural language semantics  is often formalised and computed  within  type theory 
--- as exemplified in the present issue, see section \ref{applications}.

\section{From a particular categorical interpretation to linear logic} 
\label{linear} 

While Martin-L\"of type theory discussed in previous section \ref{mtt} 
extends the expressive power of simply typed theory with dependent types, 
one can also have a closer look at the intuitionistic logical connectives 
$\&$  and $\fl$  from a computational perspective bearing in mind that an involutive negation like the one of classical logic brings many convenient and elegant properties. One thus discompose the connective of intuitonistic logic and finds a classical but constructive logic called \emph{linear logic}.\cite{Gir87,Girard2011blindspot}. 

Linear logic, present in two papers of the present issue, arose from the study of a particular  cartesian closed category that interprets proofs of intuitionistic logic, namely coherence spaces
and deserve some extra information: \emph{coherence spaces}, which are discussed in the present issue 
regarding their relation to formal ontologies, see section \ref{applications}.

A coherence space is basically a simple graph (possibly infinite, but without loops nor multiple edges)
but entities that interpret terms or proofs  in a coherence space are not its vertices but the cliques of this graph, i.e. its complete subgraphs. 
A stable map from $A$ to $B$  maps cliques of $A$ to cliques of $B$, and it enjoys certain properties: 
monotony ($a\subset b \Rightarrow F(a)\subset F(b)$), preservation of directed unions ($F(\cup\uparrow a_i)=\cup F(a_i)$) and of intersections inside a clique 
(when $a\cup a'$ is a clique of $A$ one has $F(a\cap b)=F(a)\cap F(b)$) 
The important point is that a stable function  from a coherence space $A$ to another one $B$ can  be viewed as a clique of a larger coherence space $A\fl B$, whose vertices are $(a,\beta)$ where $\beta$ is a vertex of $B$ and $a$ a \emph{finite} clique of $A$, such that $\beta\in F(a)$ and $a$ is minimal for this property --- $F$ being stable there always exists such a finite $a$,  and one of them $a$ is minimal and in fact the minimum. 
Two vertices  $(a,\beta)$ and  $(a',\beta')$ of $A\fl B$ are said to be coherent whenever 
$a\cup a'$ is a clique of $A$ implies that $\beta$ and $\beta'$ are coherent in $B$ 
and when furthermore $a$ and $a'$ are distinct so are $\beta$ and $\beta'$. 

One then has a bijection between stable functions from $A$ to $B$ and  the cliques of the coherence space $A\imp B$. There is also a product of coherence spaces, which is the disjoint union of the two graphs, plus edges between any vertex of $A$ and any vertex of $B$.

The categorical product corresponds to the rule  from $\Gamma\seq A$ and  $\Gamma\seq B$ infer 
 $\Gamma\seq A \& B$,  while the other product $\otimes$ corresponds to the rule 
 from $\Gamma\seq A$ and  $\Delta\seq B$ infer  $\Gamma,\Delta\seq A \otimes B$.
 These two conjunctions of linear logic cannot be distinguished in intuitionistic logic, because of the weakening and contraction rules given in figure \ref{curry}. 
 
  Interpreting proofs in coherence spaces, where operation on coherence spaces correspond to logical connectives lead to linear logic. Indeed, besides the product described above there is another ``product", which is not a categorical product,  
whose vertices are pairs of vertices, with an edge  between $(a,b)$ and $(a',b')$ whenever there is an edge $\{a,a'\}$ and an edge $\{b,b'\}$ in $B$ --- $a=a'$ or $b=b'$ is allowed as well, but not both since there are no loops in simple graphs. 

 Besides these pairs of conjunction, one discovers an involutive negation $A^\perp$ 
 whose simple graph is the complement of the simple graph (simple graphs contain no loop). 
 Using De Morgan laws it maps the two aforementioned conjunctions to  two disjunctions respectively written $\oplus$ (additive disjunction) and $\wp$ (multiplicative disjunction). 
 
There also exists a unary operator $!$ which relates linear logic (resource sensitive) to intuitionistic logic. 
The coherence space $!A$ is defined as follows: 
its vertices are the finite cliques of $A$, and two vertices $a$ and $a'$ of $!A$, that is to say two finite cliques of $A$, are coherent in $!A$ whenever the union $a\cup a'$ is a clique of $A$. 
This operator turns the additive conjunction $\&$ into the multiplicative conjunction $\otimes$: 
$!(A\&B)\equiv !A\otimes !B$. 

A kind  of classical implication can be defined from the $\wp$ disjunction: $A^\perp\wp B$,
it is written as $A\multimap B$, and the intuitionistic implication can be defined as: 
$$A\fl B = !A\multimap B$$ 

Linear logic was provided with a syntax, both with sequent calculus and with proofnets that identify proofs up to rule permutations. Propositional linear logic is also endowed with a truth value semantics \emph{phase semantics}. 

\emph{In this issue, the paper entitled \emph{Relational semantics for full linear logic} by Dion Coumans, Mai Gehrke,  Lorijn van Rooijen proposes a Kripke style semantics for full propositional linear logic, i.e. with, additives, multiplicatives and exponentials that allow the addition of proper axioms.}

\section{Linguistic applications: compositional semantics, lexical semantics, and knowledge representation} 
\label{applications} 

Natural language semantics is usually expressed by first order (some times higher order) logic or predicate calculus,  single sorted as Frege wished to, and computed according to Frege's compositionally principle:  the meaning of the compound is a function of the meaning of its parts. 
\cite{vBtM2010,partee02,DavisGillon2004intro}. 

This is easily implemented in lambda calculus where anything is a function, as noticed by Montague in 1970 \cite{montague:formal}. 
Typed $\lambda$-terms for representing logical formulae are usually defined out of two  base types, $\eee$ for individuals  (also known as \textbf{e}ntities) and  $\ttt$ for propositions (which have a \textbf{t}ruth value). Logical formulae can be defined in this typed $\lambda$-calculus as first observed by Church in 1940 \cite{Church1940types}, but the combination of word meaning (partial formulae) according to the syntactic structure of the sentence is due to Montague. 

This early use of lambda calculus, where formulae are viewed as typed lambda terms, can not be merged with the more familiar view of  typed lambda terms as proofs evoked in the section \ref{lambda}. The proof which such a typed lambda term corresponds to is simply the proof that the formula is well formed, e.g. that a two-place predicate is properly applied to \emph{two} individual terms of type $\eee$ and not to more or less arguements, nor to arguments of a different type etc. 
This initial vision of lambda calculus was designed for a proper handling of substitution in deductive systems \`a la Hilbert. 
One needs constants for the logical quantifiers and connectives:

\begin{center}
$
\begin{array}[t]{r|l} 
\multicolumn{2}{c}{\mbox{Logical\ connectives\ and\ quantifiers}}\\ 
\mbox{Constant} & \mbox{Type}\\ \hline 
	\textrm{\&, and} & \ttt \fl (\ttt \fl \ttt) \\ 
	\textrm{$\lor$, or} & \ttt \fl (\ttt \fl \ttt) \\ 
	\textrm{$\Rightarrow$, implies} & \ttt \fl (\ttt \fl \ttt)\\  \hline 
	\exists & (\eee \fl \ttt) \fl \ttt \\ 
	\forall & (\eee \fl \ttt) \fl \ttt \\ 
\end{array} 
$ 
\hspace{8em} 
$
\begin{array}[t]{r|l} 
\multicolumn{2}{c}{\mbox{Predicates\ and\ constants}}\\ 
\mbox{Constant} & \mbox{Type}\\ \hline 
	\mathit{defeated,\ldots} & \eee \fl (\eee \fl \ttt) \\ 
	\mathit{won, voted,\ldots} & (\eee \fl \ttt) \\ 
	\mathit{Liverpool, Leeds,\ldots} & \eee\\ 
	\mathit{\ldots} & \ldots
\end{array} 
$
\end{center} 

\noindent as well as predicates for the precise language to be described --- 
a binary predicate like $won$ has the type $\eee\fl\eee\fl\ttt$.

A small example goes as follows. Assume that the  syntactical analysis of the sentence ``\emph{Some club defeated Leeds.}" is 
\begin{center}
(some\ (club)) (defeated\ Leeds) 
\end{center} 
where the function is always the term on the left. If the semantic terms are as in the lexicon in figure \ref{semanticlexicon}, placing the semantical terms in place of the words yields a large $\lambda$-term that can be reduced: 

\begin{figure} 
\begin{center} 
\begin{tabular}{ll} \hline 
\textbf{word} &  \textbf{\itshape semantic type $u^*$}\\ 
& \textbf{\itshape  semantics~: $\lambda$-term of type $u^*$}\\ 
&  {\itshape  $x\type{v}$ the variable or constant $x$ 
is of type $v$}\\ \hline 
\textit{some} 
& $(\eee\fl \ttt)\fl ((\eee\fl \ttt) \fl \ttt)$\\ 
& $\lambda P\type{\eee\fl \ttt}\  \lambda Q\type{\eee\fl \ttt}\  
(\exists\type{(\eee\fl \ttt)\fl \ttt}\  (\lambda x\type{\eee}  (\et\type{\ttt\fl (\ttt\fl \ttt)} (P\ x) (Q\ x))))$ \\  \hline 
\textit{club}  & $\eee\fl \ttt$\\ 
& $\lambda x\type{\eee} (\texttt{club}\type{\eee\fl \ttt}\  x)$\\  \hline 
\textit{defeated} & $\eee\fl (\eee \fl \ttt)$\\ 
& $\lambda y\type{\eee}\  \lambda x\type{\eee}\  ((\texttt{defeated}\type{\eee \fl (\eee \fl \ttt)}\  x)  y)$ \\  \hline 
\textit{Leeds} &$\eee$ \\ &  Leeds 
\end{tabular}
\end{center} 
\caption{A simple semantic lexicon} 
\label{semanticlexicon}
\end{figure}

$$
\begin{array}{c} 
\Big(\big(\lambda P\type{\eee\fl \ttt}\ \lambda Q\type{\eee\fl \ttt}\  (\exists\type{(\eee\fl \ttt)\fl \ttt}\  (\lambda x\type{\eee}  (\et (P\ x) (Q\ x))))\big)
\big(\lambda x\type{\eee} (\texttt{club}\type{\eee\fl \ttt}\  x)\big)\Big) \\ 
\Big(
\big(\lambda y\type{\eee}\  \lambda x\type{\eee}\  ((\texttt{defeated}\type{\eee\fl (\eee\fl \ttt)}\  x)  y)\big)\ Leeds\type{\eee}\Big)\\ 
\multicolumn{1}{c}{\downarrow \beta}\\ 
\big(\lambda Q\type{\eee\fl \ttt}\  (\exists\type{(\eee\fl \ttt)\fl \ttt}\  (\lambda x\type{\eee}  (\et\type{\ttt\fl (\ttt\fl \ttt)}  
(\texttt{club}\type{\eee\fl \ttt}\  x) (Q\ x))))\big)\\ 
\big(\lambda x\type{\eee} \ ((\texttt{defeated}\type{\eee\fl (\eee \fl \ttt)}\  x)  Leeds\type{\eee})\big)\\ 
\multicolumn{1}{c}{\downarrow \beta}\\ 
\big(\exists\type{(\eee\fl \ttt)\fl \ttt}\  (\lambda x\type{\eee}  (\et (\texttt{club}\type{\eee\fl \ttt}\  x) ((\texttt{defeated}\type{\eee\fl (\eee\fl \ttt)}\  x)  Leeds\type{\eee})))\big)
\end{array}
$$ 

This $\lambda$-term of type $\ttt$ that can be called the \emph{logical form} of the sentence, represents the following formula of predicate 
calculus (admittedly more pleasant to read): 

$$\exists x:\eee\  (\texttt{club}(x)\ \et\ \mathit{defeated}(x,Leeds))$$

Thus defined, computational semantics looks a miracle. But this miracle does not account for selectional restriction (a \ma{book} cannot \ma{bark}), does not tell anything of the relations between predicates (a \ma{book} can be \ma{read}), nor of the correspondence of predicates with concepts (a book is both an informational and a physical entity),  nor of the relation between the logical language for semantics and world knowledge and ontologies. 
Issues such as selectional restriction have been much 
discussed in type theoretical framework  e.g. in 
\cite{Cooper2006dag,asher-typedriven,BMRjolli,Asher2011wow,Luo2011lacl}. 

\emph{In his paper \emph{Selectional Restrictions, Types and Categories}, Nicholas Asher  discusses the introduction of types as objects of a category  (instead of a single $\eee$) endowed with categorical operations that go beyond the logical ones: he is thus able to correctly account for selectional restriction, felicitous and infelicitous 
copredications.} 

\emph{This framework in which one composes functions to compute meanings and ultimately the truth values in a given model is given a categorical formalisation by Anne Preller in her paper \emph{Natural Language Semantics in Biproduct Dagger Categories}. Starting with pregroup grammars defined in an algebraic structure that can be viewed as non commutative  linear logic in which $\wp=\otimes$, she computes the meaning categories with extra operations, as finite dimensional vector spaces or two sorted functions over finite sets. She can then account of meaning with  the usual set theoretic interpretations or with bags of words vectors used for text classification. }  

An underlying question, in relation to the interpretation of common nouns and to the bags of words approach is what a concept or a base type is, and the interrelations between base types. 
For instance one can say \ma{a cat sleeps}, because \ma{animals} can \ma{sleep} and \ma{cats} are  \ma{animals}. 
In the kind of model used in semantics ontological relations are quite common, but here is no specific structure to represent these ontological relations. 

\emph{The paper entitled \emph{Formal ontologies and coherent spaces} by Michel Abrusci, Christophe Fouquer\'e and Marco Romano exploits the structure of coherent spaces that we discussed in section \ref{linear} to give a formal account of the relation between entities and concepts, encoding properties of formal ontologies.}

\paragraph{Thanks:} We are indebted to Mai Gehrke, Alain Lecomte, 
Anne Preller,  Steve Vickers, Gilles Z\'emor, and  especially to Fran\c{c}ois Lamarche and Thomas Streicher for a quick but efficient rereading of this introduction, which they all contributed to improve. Remaining errors and confusions are ours. 

\bibliography{bigbiblio} 

\providecommand{\bysame}{\leavevmode\hbox to3em{\hrulefill}\thinspace}
\providecommand{\MR}{\relax\ifhmode\unskip\space\fi MR }
\providecommand{\MRhref}[2]{%
  \href{http://www.ams.org/mathscinet-getitem?mr=#1}{#2}
}
\providecommand{\href}[2]{#2}
\begin{thebibliography}{10}

\bibitem{SGA4}
Michael Artin, Alexandre Grothendieck, and Jean-Louis Verdier (eds.),
  \emph{Th{\'e}orie des topos et cohomologie {\'e}tale des sch{\'e}mas},
  Lecture Notes in Mathematics, no. 269 \& 270, {S}pringer-{V}erlag, 1972,
  S{\'e}minaire de G{\'e}om{\'e}trie Alg{\'e}brique du Bois-Marie (1963-1964).

\bibitem{Asher2011wow}
Nicholas Asher, \emph{Lexical meaning in context -- a web of words}, Cambridge
  University press, 2011.

\bibitem{asher-typedriven}
Nicolas Asher, \emph{A type driven theory of predication with complex types},
  Fundamenta Informaticae \textbf{84} (2008), no.~2, 151--183.

\bibitem{ACV2013}
Steve Awodey, Thierry Coquand, and Vladimir Voevodsky (eds.), \emph{Homotopy
  type theory: univalent foundations of mathematics}, Institute of Advance
  Studies, Princeton, 2013.

\bibitem{BMRjolli}
{C}hristian {B}assac, {B}runo {M}ery, and {C}hristian {R}etor{\'e},
  \emph{{T}owards a {T}ype-{T}heoretical {A}ccount of {L}exical {S}emantics},
  {J}ournal of {L}ogic {L}anguage and {I}nformation \textbf{19} (2010), no.~2,
  229--245 ({A}nglais), \url{http://hal.inria.fr/inria-00408308/}.

\bibitem{BC2004coq}
Yves Bertot and Pierre Cast\'eran, \emph{Interactive theorem proving and
  program development. coq'art: The calculus of inductive constructions}, Texts
  in Theoretical Computer Science, Springer Verlag, 2004.

\bibitem{BoileauJoyal1981}
Andr\'e Boileau and Andr\'e Joyal, \emph{La logique des topos}, Journal of
  Symbolic Logic \textbf{46} (1981), no.~1, 6--16.

\bibitem{Borceux1994hdbk3}
Francis {Borceux}, \emph{Handbook of categorical algebra. 3: {C}ategories of
  sheaves.}, Cambridge: Cambridge Univ. Press, 1994 (English).

\bibitem{Cartier1979bourbaki}
Pierre Cartier, \emph{Logique, cat{\'e}gories et faisceaux (d'apr{\`e}s {F}.
  {L}awvere et {M}. {T}ierney) --- expos{\'e} 513}, S{\'e}minaire Bourbaki,
  30{\`e}me ann{\'e}e (1977/1978), Lecture Notes in Mathematics, vol. 710,
  Springer, 1979, pp.~123--146.

\bibitem{Church1940types}
Alonzo Chruch, \emph{A formulation of the simple theory of types}, Journal of
  Symbolic Logic \textbf{5} (1940), 56--68.

\bibitem{CIardelli2011tacl}
Ivano~Alessandro Ciardelli, \emph{A canonical model for presheaf semantics},
  Tech. report, INRIA, 2011, Presented at \emph{Topology, Algebra and
  Categories in Logic (TACL) 2011, Marseilles}.

\bibitem{cohen1966set}
Paul~J. Cohen, \emph{Set theory and the continuum hypothesis}, Mathematics
  lecture note series, W. A. Benjamin, 1966.

\bibitem{Cooper2006dag}
Robin Cooper, \emph{A record type theoretic account of copredication and
  dynamic generalized quantification}, Philosophical communications, web series
  \textbf{35} (2006), 75--88.

\bibitem{DavisGillon2004intro}
Steven Davis and Brendan~S. Gillon, \emph{Introduction (to ``semantics, a
  reader")}, Semantics: a reader (Steven Davis and Brendan~S. Gillon, eds.),
  Oxford University Press, New York, 2004, pp.~1--130.

\bibitem{Gen34}
Gehrard Gentzen, \emph{Untersuchungen {\"u}ber das logische {S}chlie\ss en
  {I}}, Mathematische {Z}eitschrift \textbf{39} (1934), 176--210, Traduction
  Fran\c caise de R.~Feys et J.~Ladri\`ere: Recherches sur la d\'eduction
  logique, Presses Universitaires de France, Paris, 1955.

\bibitem{Gen34b}
\bysame, \emph{Untersuchungen {\"u}ber das logische {S}chlie\ss en {II}},
  Mathematische Zeitschrift \textbf{39} (1934), 405--431, Traduction fran\c
  caise de J.~Ladri\`ere et R.~Feys: Recherches sur la d\'eduction logique,
  Presses Universitaires de France, Paris, 1955.

\bibitem{Gir87}
Jean-Yves Girard, \emph{Linear logic}, Theoretical Computer Science \textbf{50}
  (1987), no.~1, 1--102.

\bibitem{Girard2011blindspot}
\bysame, \emph{The blind spot -- lectures on logic}, European Mathematical
  Society, 2011.

\bibitem{Girard2012normativity}
\bysame, \emph{Normativity in logic}, Epistemology versus Ontology (Peter
  Dybjer, Sten Lindstr{\"o}m, Erik Palmgren, and G{\"o}ran Sundholm, eds.),
  Logic, Epistemology, and the Unity of Science, vol.~27, Springer, 2012,
  pp.~243--263.

\bibitem{GLT88}
Jean-Yves Girard, Yves Lafont, and Paul Taylor, \emph{Proofs and types},
  Cambridge Tracts in Theoretical Computer Science, no.~7, Cambridge University
  Press, 1988.

\bibitem{gpdtt}
Martin Hofmann and Thomas Streicher, \emph{The groupoid interpretation of type
  theory}, Twenty-five years of Martin-L\"of type theory (G.~Sambin and
  J.~Smith, eds.), Oxford University Press, 1998, pp.~83--111.

\bibitem{How69}
William~A. Howard, \emph{The formulae-as-types notion of construction}, To H.B.
  Curry: Essays on Combinatory Logic, $\lambda$-calculus and Formalism
  (J.~Hindley and J.~Seldin, eds.), Academic Press, 1980, pp.~479--490.

\bibitem{KK86}
William Kneale and Martha Kneale, \emph{The development of logic},
  3$^{\textrm{rd}}$ ed., Oxford University Press, 1986.

\bibitem{Kripke59jsl}
Saul Kripke, \emph{A completeness theorem in modal logic}, J. Symb. Log.
  \textbf{24} (1959), no.~1, 1--14.

\bibitem{LS86}
Joachim Lambek and Phil~J. Scott, \emph{Introduction to higher order
  categorical logic}, Cambridge Studies in Advanced Mathematics, vol.~7,
  Cambridge University Press, 1986.

\bibitem{Lawvere1970congres}
Francis~William Lawvere, \emph{Quantifiers and sheaves}, Actes du congr{\`e}s
  international des math{\'e}maticiens (Marcel Berger, Jean Dieudonn{\'e}, Jean
  Leray, Jacques-Louis Lions, Paul Malliavin, and Jean-Pierre Serre, eds.),
  Gauthier-Villars, 1970, pp.~329--334.

\bibitem{Lawvere72}
Francis~William Lawvere (ed.), \emph{Toposes, algebraic geometry, and logic},
  Lecture Notes in Computer Science, vol. 274, Springer, Berlin, Heidelberg,
  New York, 1972.

\bibitem{Luo2011lacl}
Zhaohui Luo, \emph{Contextual analysis of word meanings in type-theoretical
  semantics}, in Pogodalla and Prost \cite{LACL2011}, pp.~159--174.

\bibitem{Mac71}
Saunders Mac~Lane, \emph{Categories for the working mathematician},
  Springer-Verlag, Berlin, 1971.

\bibitem{MacLaneMoerdijk1992}
Saunders Mac~Lane and Ieke Moerdijk, \emph{{Sheaves in geometry and logic: a
  first introduction to topos theory.}}, {Universitext. New York etc.:
  Springer-Verlag. xii, 627 p. }, 1992 (English).

\bibitem{ML73}
Per Martin-L{\"o}f, \emph{An intuitionistic theory of types: predicative part},
  Logic Colloquium, Bristol 19731973, {S}tudies in {L}ogic and {F}oundations of
  {M}athematics, vol.~80, North Holland, 1975, pp.~73--118.

\bibitem{ML84}
\bysame, \emph{Intuitionistic type theory}, Bibliopolis, Napoli, 1984, (Lecture
  Notes by G.~Sambin).

\bibitem{Mellies2012lics}
Paul-Andr{\'e} Melli{\`e}s, \emph{Game semantics in string diagrams}, LICS,
  IEEE, 2012, pp.~481--490.

\bibitem{montague:formal}
Richard Montague, \emph{English as a formal language}, Linguaggi nella Societa
  e nella Tecnica (Bruno Visentini, ed.), Edizioni di Communit{\`a}, Milan,
  Italy, 1970, (Reprinted in R. Thomason (ed) \emph{The collected papers of
  Richard Montague} Yale University Press, 1974.), pp.~189--224.

\bibitem{MoortgatMoot2013grishin}
Michael Moortgat and Richard Moot, \emph{Proof nets for the {L}ambek-{G}rishin
  calculus}, Quantum Physics and Linguistics: A Compositional, Diagrammatic
  Discourse (Chris Heunen, Mehrnoosh Sadrzadeh, and Edward Grefenstette, eds.),
  Oxford University Press, 2013, pp.~283--320.

\bibitem{LACL2011}
Sylvain Pogodalla and Jean-Philippe Prost (eds.), \emph{Logical aspects of
  computational linguistics - 6th international conference, lacl 2011,
  montpellier, france, june 29 - july 1, 2011. proceedings}, LNCS, vol. 6736,
  Springer, 2011.

\bibitem{Pollard2011lacl}
Carl Pollard, \emph{Are (linguists') propositions (topos) propositions?}, in
  Pogodalla and Prost \cite{LACL2011}, pp.~205--218.

\bibitem{partee02}
Paul Portner and Barbara~H. Partee (eds.), \emph{Formal semantics: The
  essential readings}, Blackwell Publishers, 2002.

\bibitem{quillen1967homotopical}
Daniel~G. Quillen, \emph{Homotopical algebra}, Lecture Notes in Mathematics,
  Springer, 1967.

\bibitem{Streicher2006uppsala}
Thomas Streicher, \emph{Identity types and weak omega-groupoids}, Identity
  Types - Topological and Categorical Structure, 2006, Uppsala talk available
  from \url{http://www.mathematik.tu-darmstadt.de/~streicher/}.

\bibitem{TD88v2}
Anne Troelstra and Dirk van Dalen, \emph{Constructivism in mathematics (vol.
  2)}, Studies in Logic and the Foundations of Mathematics, vol. 123,
  North-Holland, 1988.

\bibitem{vBtM2010}
Johan van Benthem and Alice ter Meulen, \emph{Handbook of logic and language},
  2nd ed., Elsevier insights, Elsevier Science, 2010.

\bibitem{vanDalen2013}
Dirk van Dalen, \emph{Logic and structure}, fifth ed., Universitext,
  Springer-Verlag, 2013.

\bibitem{Voevodsky2006lambda}
Vladimir Voevodsky, \emph{A very short note on homotopy lambda calculus}, Tech.
  report, Institute for Advanced Study, Princeton, 2006.

\end{thebibliography}
\bibliographystyle{amsplain} 
\end{document}